\documentclass[12pt]{article}

\usepackage{amsmath,amsthm,amssymb,amsfonts,amscd, mathrsfs, mathtools, dsfont}
\usepackage[hyperindex,breaklinks]{hyperref}
\usepackage{cleveref}
\usepackage{tikz}
\usepackage{tikz-cd}
\usepackage{cancel}
\usepackage{fullpage}
\usepackage{enumerate}
\usepackage[shortlabels]{enumitem}
\usepackage{makecell}
\usetikzlibrary{matrix,arrows,decorations.pathmorphing}

\theoremstyle{definition}
\newtheorem{theorem}{Theorem}[section]

\newtheorem{lemma}[theorem]{Lemma}
\newtheorem{proposition}[theorem]{Proposition}

\newtheorem{definition}[theorem]{Definition}

\newtheorem*{acknowledgements}{Acknowledgements}

\numberwithin{equation}{section}
\numberwithin{theorem}{section}

\definecolor{darkgreen}{rgb}{0,0.5,0}
\definecolor{darkblue}{rgb}{0,0.1,0.5}

\hypersetup{
    colorlinks=true, % make the links colored
    linkcolor=darkblue, % color TOC links in blue
    urlcolor=blue, % color URLs in red
    linktoc=all, % 'all' will create links for everything in the TOC
    citecolor=darkgreen
    }

\newcommand{\Hom}{\mathrm{Hom}}

\newcommand{\Ker}{\mathrm{Ker}}

\newcommand{\UH}[1]{U^{H}_q(#1)}
\newcommand{\UHbar}[1]{\overline{U}_q^{H}(#1)}

\newcommand{\mfb}{\mathfrak{b}}

\newcommand{\mfg}{\mathfrak{g}}
\newcommand{\mfh}{\mathfrak{h}}

\newcommand{\mcC}{\mathcal{C}}

\newcommand{\mcM}{\mathcal{M}}

\newcommand{\mcO}{\mathcal{O}}

\newcommand{\mbbC}{\mathbb{C}}

\newcommand{\mbbN}{\mathbb{N}}

\newcommand{\mbbZ}{\mathbb{Z}}

\newcommand{\w}{\mathsf{w}}

\newcommand{\Verma}[1]{M_{#1}^{\bar{0}}}
\newcommand{\DVerma}[1]{\check{M}_{#1}^{\bar{0}}}
\newcommand{\GVerma}[2]{M_{#1}^{#2}}
\newcommand{\DGVerma}[2]{\check{M}_{#1}^{#2}}
\newcommand{\irred}[1]{L_{#1}}
\newcommand\doi[2]{\href{http://dx.doi.org/#1}{#2}}

\newcommand{\Proj}[2]{P_{#1}^{#2}}

\title{BGG Reciprocity for Generalized Weight Modules of Unrolled Restricted Quantum Groups at Roots of Unity}
\author{Matthew Rupert}
\date{}

\begin{document}
\maketitle

\begin{abstract}We consider the category of generalized weight modules over the unrolled restricted quantum group $\UHbar{\mfg}$ of a finite-dimensional simple complex Lie algebra $\mfg$ at root of unity $q$. Although this category does not admit projective modules, it is filtered by subcategories with enough projectives. We show that the projective covers of irreducible modules in these subcategories satisfy a variant of BGG reciprocity and give a generators and relations description when $\mfg=\mathfrak{sl}_2$.\end{abstract}

\tableofcontents

\setlength{\parskip}{\baselineskip}%

\section{Introduction}
The unrolled quantum group $\UH{\mfg}$ associated to a finite-dimensional semisimple complex Lie algebra $\mfg$ is an extension of the non-restricted specialization $U_q(\mfg)$ of De Concini and Kac \cite{DCK}. It can be obtained by taking a smash product with the universal enveloping algebra of the Cartan subalgebra $\mfh$ of $\mfg$, $\UH{\mfg}:= U_q(\mfg) \rtimes U(\mfh)$. Inspired by previous work of Ohtsuki \cite{O}, a quotient of this algebra called the unrolled restricted quantum group $\overline{U}_q^H(\mfg)$ was introduced in \cite{GPT1} for $\mfg=\mathfrak{sl}_2$ and studied at even roots of unity as an example for constructing link invariants from its representation theory (see also \cite{GPT2}). Unrolled restricted quantum groups have since found substantial applications to quantum topology as their category of weight modules $\mathrm{Rep}_{wt}\overline{U}_q^H(\mathfrak{g})$ can be used to construct decorated topological quantum field theories \cite{BCGP,D,DGP}. A second application of these algebras, and the application of primary concern for this article, is the logarithmic Kazhdan-Lusztig correspondence for vertex operator algebras.

The study of equivalences between categories of modules over quantum groups and vertex operator algebras began with the pioneering work of Kazhdan and Lusztig \cite{KL1,KL2,KL3,KL4}. Proved therein was a braided equivalence between categories of modules over simple affine vertex operator algebras $L_k(\mfg)$ with $k+h^{\vee} \in \mathbb{Q}_{\geq 0}$ where $h^{\vee}$ is the dual Coxeter number and a corresponding quantum group. It is generally expected that many vertex operator algebras (in particular, those which admit free field realisations) should admit categories of modules which are equivalent to some category of modules over a corresponding quantum group. This expectation is broadly referred to as the Kazhdan-Lusztig correspondence, and logarithmic when the categories involved are non-semisimple. The study of logarithmic Kazhdan-Lusztig equivalences was first explored for the Triplet vertex operator algebra $\mathcal{W}(p)$ ($p \geq 2$) in the work of Feigin et al \cite{FGST1,FGST2} where a ribbon equivalence with a certain category of modules of the restricted quantum group $\overline{U}_q(\mathfrak{sl}_2)$ was conjectured. This conjecture was studied \cite{NT,KS} and refined \cite{CGR} to an equivalence between modules over $\mathcal{W}(p)$ and a quasi-Hopf modification $\tilde{U}_q(\mathfrak{sl}_2)$ of $U_q(\mathfrak{sl}_2)$. This refined conjecture was eventually proven for $p=2$ in \cite{CLR} and for all $p \geq 2$ in \cite{GN}.

The Triplet $\mathcal{W}(p)$ contains a sub-VOA called the Singlet $\mathcal{M}(p)$ which was expected to admit a category of modules which is braided equivalent to the category of weight modules $\mathrm{Rep}_{wt}\UHbar{\mfg}$ over the unrolled restricted quantum group at $2p$-th root of unity $q$. This conjecture was motivated in \cite{CMR,CGR} and connections to unrolled restricted quantum groups were established for simple current extensions of vertex algebras related to the Singlet in \cite{ACKR,CRR}. The conjecture was further refined to a particular category $\mcO^T_{\mathcal{M}(p)}$ of Singlet modules in \cite{CMY} and the braided equivalence $\mathrm{Rep}_{wt}\UHbar{\mfg} \cong \mcO^T_{\mathcal{M}(p)}$ has now been proved in \cite{CLR2}. Both of these categories consist of modules which decompose as a direct sum of weight spaces over some operators. We can consider instead the categories $\mcO_{\mcM(p)}^{T,Gen}, \mathrm{Rep}_{Gen}\UHbar{\mfg}$ of modules which decompose as direct sums of generalized weight spaces. It is expected that the braided equivalence of weight module categories proved in \cite{CLR2} should extend to a braided equivalence 
\[ \mathrm{Rep}_{Gen}\UHbar{\mathfrak{sl}_2} \cong \mcO_{\mcM(p)}^{T,Gen}\]
of categories of generalized weight modules. The focus of this article is to investigate the abelian structure of $\mathrm{Rep}_{Gen}\UHbar{\mfg}$. We show that this category is filtered by subcategories which have enough projectives and that projective covers satisfy a variant of BGG reciprocity. We further show that projective and tilting modules coincide and use this result to give a generators and relations description of projective covers of irreducible modules when $\mfg=\mathfrak{sl}_2$.

\begin{acknowledgements}
The author is supported by the Pacific Institute for the Mathematical Sciences (PIMS) postdoctoral fellowship program.
\end{acknowledgements}

\subsection{Results}
Let $\mfg$ be a finite-dimensional simple complex Lie algebra and $\UHbar{\mfg}$ the associated unrolled restricted quantum group (Definitions \ref{Def:unrolled} \& \ref{Def:unrolledrestricted}). Let $\mcC:=\mathrm{Rep}_{Gen}\UHbar{\mfg}$ denote the category of finite-dimensional generalized $\UHbar{\mfg}$-weight modules and $\mcC_{\bar{m}}$ the subcategory consisting of modules whose generalized weights have maximal degree $\bar{m} \in \mathbb{N}^n$ where $n=\mathrm{rank}(\mfg)$ (see Section \ref{Sec:genweight} for details). Each subcategory $\mcC_{\bar{m}}$ contains all irreducible modules in $\mcC$ and these irreducible modules admit projective covers in $\mcC_{\bar{m}}$ (Theorem \ref{Thm:projexist}):
\begin{theorem}
$\mcC_{\bar{m}}$ has enough projectives.
\end{theorem}
Although $\mcC$ does not itself admit projective modules (Theorem \ref{Thm:projexist}), every module $M \in \mcC$ satisfies $M \in \mcC_{\bar{m}}$ for some $\bar{m} \in \mathbb{N}^n$. The abelian structure of $\mcC$ can therefore be understood through the structure of projective modules in $\mcC_{\bar{m}}$ for each choice of $\bar{m} \in \mathbb{N}^n$. For any weight $\lambda \in \mfh^*$ where $\mfh$ denotes the Cartan subalgebra of $\mfg$, we introduce the generalized Verma module $\GVerma{\lambda}{\bar{m}}$ of weight $\lambda$ and degree $\bar{m}$ in Definition \ref{Def:genverma}. The irreducible modules in $\mcC$ are the irreducible quotients $\irred{\lambda}$ of the generalized Verma modules (of any degree) with the same highest weight. We show (Lemma \ref{Prop:projfilt}) that every projective module in $\mcC_{\bar{m}}$ admits a degree $\bar{m}$ standard filtration (Definition \ref{Def:standard}). That is, a filtration whose successive quotients consist of generalized Verma modules $\GVerma{\lambda}{\bar{m}}$ with degree $\bar{m}$. If we denote by $(P,\GVerma{\lambda}{\bar{m}})$ the multiplicity of $\GVerma{\lambda}{\bar{m}}$ in the degree $\bar{m}$ standard filtration of a projective module $P \in \mcC_{\bar{m}}$, then we prove the following variant of BGG reciprocity (Theorem \ref{Thm:BGG}):
\begin{theorem}
The projective cover $\Proj{\lambda}{\bar{m}}$ of $\irred{\lambda}$ in $\mcC_{\bar{m}}$ satisfies 
\[ (\Proj{\lambda}{\bar{m}},\GVerma{\mu}{\bar{m}})=[\GVerma{\mu}{\bar{0}},\irred{\lambda}]. \]
\end{theorem}
Here, $[\GVerma{\mu}{\bar{0}},\irred{\lambda}]$ denotes the multiplicity of $\irred{\lambda}$ in the Jordan-Holder filtration of the (weight) Verma module $\GVerma{\mu}{\bar{0}}$. We call a module $M \in \mcC_{\bar{m}}$ tilting if it admits both standard and costandard filtrations of degree $\bar{m}$ where a costandard filtration is a filtration whose successive quotients are the images of generalized Verma modules under the duality functor introduced in Definition \ref{Def:dual}. We have the following characterization of projective modules (Proposition \ref{Prop:tiltingproj}):
\begin{proposition}\label{Prop:introtilting}
A module $M \in \mcC_{\bar{m}}$ is projective iff it is tilting.
\end{proposition}
In Section \ref{Sec:projcoversl2} we restrict our attention to the case $\mfg=\mathfrak{sl}_2$ and apply Proposition \ref{Prop:introtilting} to determine a generators and relations description of projective covers. Let $\ell=\mathrm{ord}(q)$ and let $r=\ell$ if $\ell$ is odd and $r=\ell/2$ if $\ell$ is even. The irreducible modules of $\UHbar{\mathfrak{sl}_2}$ fall into two families \cite[Sections 5 \& 10]{CGP}:
\begin{itemize}
\item $L_i \otimes \mathbb{C}_{\frac{k \ell}{2}}$ for $i=0,...,r-2,k \in \mathbb{Z}$.
\item $M_{\alpha}$ for $\alpha+1 \in \ddot{\mathbb{C}}$.
\end{itemize}
where $\ddot{\mathbb{C}}:=\begin{cases}
(\mbbC \setminus \mathbb{Z}) \cup r\mathbb{Z} & \text{ if $\ell$ is even}\\
(\mbbC \setminus \frac{1}{2}\mathbb{Z}) \cup \frac{r}{2}\mathbb{Z} & \text{ if $\ell$ is odd}.
\end{cases}
$\\

$\mathbb{C}_{\frac{k \ell}{2}}^H$ is a one-dimensional module with weight $\frac{k\ell}{2}$, $L_i$ is an $(i+1)$-dimensional irreducible module and $M_{\alpha}$ is an irreducible Verma module of highest weight $\alpha$. The projective cover of $M_{\alpha}$ in $\mcC_{m}$ (recall now that $n=\mathrm{rank}(\mathfrak{sl}_2)=1$) is the generalized Verma module $\GVerma{\alpha}{m}$ of degree $m$ with the same highest weight. Let $i \in \{0,1,...,r-2\},m \in \mathbb{N}$ and let $\Proj{i}{m}$ denote the $\UHbar{\mathfrak{sl}_2}$-module with basis given by the $2(m+1)r$ vectors
\[\w_{i-2k,s}^T,\quad \w_{i-2k,s}^S, \quad \w_{j-r-2k',s}^L, \quad \w_{r-j+2k',s}^R\] with $k \in \{0,...,i\}$, $k' \in \{0,...,j\}$, and $s\in \{0,...,m\}$ satisfying Equations \eqref{eq:proj1} - \eqref{eq:proj14}. Note that in Equations \eqref{eq:proj1} - \eqref{eq:proj14} we have adopted the presentation for $\UHbar{\mathfrak{sl}_2}$ given in Equations \eqref{eq:sl21}-\eqref{eq:sl22}. The module $\Proj{i}{0}$ was shown to be projective in $\mcC_{0}$ in \cite[Proposition 6.2]{CGP} by showing that any surjection onto $\Proj{i}{0}$ splits. This argument relies on the fact that weight spaces are degree zero. We show that $\Proj{i}{m}$ is projective in $\mcC_m$ by showing that it is tilting and obtain the following result (Theorem \ref{Thm:projcover}):
\begin{theorem}
$\Proj{i}{m} \otimes \mathbb{C}_{\frac{k \ell}{2}}$ is the projective cover of $L_i \otimes \mathbb{C}_{\frac{k \ell}{2}}$ in $\mcC_m$.
\end{theorem}

\section{Preliminaries}\label{Sec:prelim}

Let $\mfg$ be a simple finite-dimensional complex Lie algebra with Cartan matrix $A=(a_{ij})_{1 \leq i,j, \leq n}$ and Cartan subalgebra $\mfh$. Denote the rank of $\mfg$ by $n$ and let $\Delta=\{\alpha_1,...,\alpha_n\} \subset \mfh^*$ be the set of simple roots of $\mfg$, $\Delta^+$ ($\Delta^-$) the set of positive (negative) roots, and $L_R=\oplus_{k=1}^n \mbbZ \alpha_k$ the integer root lattice. Let $\{H_1,...,H_n\}$ denote the basis of $\mfh$ such that $\alpha_j(H_i)=a_{ij}$ and let $\langle -,- \rangle$ be the form defined by $\langle \alpha_i,\alpha_j\rangle = d_ia_{ij}$ where $d_i=\langle \alpha_i,\alpha_i \rangle/2$ normalized such that short roots have length $2$. Finally, let $L_W=\oplus_{k=1}^n \mbbZ w_k$ be the weight lattice generated by the fundamental weights i.e. the dual basis $\{\omega_1,...,\omega_n\} \subset \mfh^*$ of $\{d_1H_1,...,d_nH_n\} \subset \mfh$. 
\begin{definition}\label{Def:unrolled}
The unrolled quantum group associated to $L$ and $q \in \mbbC^{\times}$ such that $\mathrm{ord}(q^2)>d_i$ for all $i \in \{1,...,n\}$ is the $\mbbC$-algebra with generators $K_{\gamma},X_{\pm i}, H_i$ such that $\gamma \in L$ and $i\in \{1,.,.,n\}$ with relations
\begin{align}
\label{eqKX}K_0=1, \quad K_{\gamma_1}K_{\gamma_2}=K_{\gamma_1+\gamma_2},\quad K_{\gamma}X_{\pm j}K_{-\gamma}=q^{\pm \langle \gamma,\alpha_j \rangle}X_{\pm j}, \\
\label{eqH} [H_i,H_j]=0, \quad [H_i,K_{\gamma}]=0, \qquad \quad [H_i,X_{\pm j}]=\pm a_{ij} X_{\pm j }, \quad \\
\label{eqX} \quad  [X_i,X_{-j}]=\delta_{i,j}\frac{K_{\alpha_j}-K_{\alpha_j}^{-1}}{q_j-q_j^{-1}},\qquad \qquad \qquad \\
\label{serre} \sum\limits_{k=0}^{1-a_{ij}} (-1)^k \binom{1-a_{ij}}{k}_{q_i} X^k_{\pm i}X_{\pm j} X_{\pm i}^{1-a_{ij}-k}=0 \qquad \text{if $ i \not = j$}.
\end{align}
There is a Hopf-algebra structure on $U^H_{q}(\mfg)$ with coproduct $\Delta$, counit $\epsilon$, and antipode $S$ defined by
\begin{align}
\label{coK}\Delta(K_{\gamma})&=K_{\gamma} \otimes K_{\gamma}, & \epsilon(K_{\gamma})&=1,&  S(K_{\gamma})&=K_{-\gamma}\\
\label{coXi}\Delta(X_i)&=1 \otimes X_i + X_i \otimes K_{\alpha_i} & \epsilon(X_i)&=0, & S(X_i)&=-X_iK_{-\alpha_i},\\
\label{coX-i}\Delta(X_{-i})&=K_{-\alpha_i} \otimes X_{-i}+X_{-i} \otimes 1, & \epsilon(X_{-i})&=0, & S(X_{-i})&=-K_{\alpha_i}X_{-i}.,\\
\label{coH} \Delta(H_i)&=1\otimes H_i + H_i \otimes 1,& \quad \epsilon(H_i)&=0, &\quad S(H_i)&=-H_i.
\end{align}
\end{definition}

The subalgebra generated by $X_{\pm k},K_{\gamma}$ is the De Concini - Kac quantum group $U_q(\mfg)$. It can be shown by induction on $s$ that \cite[Equation 1.3.3]{DCK}
\[[X_i,X_{-j}^s]=\delta_{ij}[s]_iX_{-i}^{s-1}[K_i;d_i(1-s)]\]
where $[K_i;n]=(K_iq^n-K_i^{-q}q^{-n})/(q_i-q_i^{-1})$. Let $d_{\alpha}=\frac{1}{2}\langle \alpha,\alpha \rangle$ and $r_{\alpha}:=r/\mathrm{gcd}(d_{\alpha},r)$, then $[r_{\alpha_i}]_i=0$ so $[X_i,X_{-j}^{r_{\alpha_i}}]=0$ for all $i,j$.  Let $\beta_1,...,\beta_N$ be a total ordering on the set of positive roots $\Delta^+$ of $\mfg$ and $X_{\pm \beta_k} \in U_q(\mfg)$ the associated root vectors constructed through the action of the braid group as in \cite[Subsection 8.1 and 9.1]{CP}. Applying the braid group action then gives $[X_{\beta},X_{-\alpha}^{r_{\alpha}}]=0$ for all $\alpha,\beta \in \Delta^+$. It can be shown as in \cite[Subsection 3.1]{R} that the ideal generated by $\{X_{\pm i}^{r_{\alpha_i}}\}_{i=1}^n$ is a Hopf ideal. We quotient by this ideal to obtain the unrolled restricted quantum group:
\begin{definition}\label{Def:unrolledrestricted}
The unrolled restricted quantum group $\UHbar{\mfg}$ is defined as the quotient of $\UH{\mfg}$ by the Hopf ideal generated by $\{X_{\pm i}^{r_{\alpha_i}}\}_{\alpha_i \in \Delta}$.
\end{definition}
We will denote by $\UHbar{\mfg}^0,\UHbar{\mfg}^{\pm}$ the subalgebras generated by $\{ K_{\gamma},H_i \, | \, \gamma \in L, i=1,...,n \}$ and $\{X_{\pm i} \, | \, i=1,...,n \}$ respectively. It follows from the braid relations that $X_{\pm \alpha}^{r_{\alpha}}=0$ for all positive roots $\alpha \in \Delta^+$ and it now follows from the PBW theorem for $U_q(\mfg)$ \cite{CP} that 
\begin{equation}\label{eq:basis} \{X_{\pm \beta_1}^{t_1}X_{\pm \beta_2}^{t_2} \cdots X_{\pm \beta_N}^{t_N} \; | \; 0 \leq t_i \leq r_{\beta_i} \}\end{equation} 
 is a vector space basis of $\UHbar{\mfg}^{\pm}$. 

It is well known that $\overline{U}_q(\mfg)$ admits an automorphism which swaps $X_j$ with $X_{-j}$ and inverts $K_j$ (see \cite[Lemma 4.6]{Ja}, for example). This automorphism can be extended to an automorphism $\omega:\UHbar{\mfg} \to \UHbar{\mfg}$ by defining
\begin{equation}\label{eq:aut} \omega(X_{\pm j})=X_{\mp j}, \quad \omega(K_{\gamma})=K_{-\gamma}, \quad \omega(H_j)=-H_j.\end{equation}
We will require this automorphism to construct a duality functor analogous to that which is used to study Category $\mcO$ for Lie algebras. 

\section{Generalized Weight Modules}\label{Sec:genweight}

For the remainder of the article, we fix $q=e^{\frac{\pi \mathsf{i}}{\ell}}$ with $\ell \in \mathbb{N}$ such that $\mathrm{ord}(q^2)>d_j$ for all $j \in \{1,...,n\}$. Given a $\UHbar{\mfg}$-module $V$, a weight $\lambda \in \mfh^*$ and $\bar{m}=(m_1,...,m_n) \in \mathbb{N}^n$, we define the subspace $V(\lambda)_{\bar{m}}$ by
\[ V(\lambda)_{\bar{m}}:=\{ v \in V \; | \; (H_j-\lambda(H_j))^{m_j+1}v=0 \, \mathrm{and} \, (H_j-\lambda(H_j))^{m_j}v\not =0 \, \mathrm{for} \, j=1,...,n\}. \]
If $V(\lambda)_{\bar{m}} \not = 0$ for some $\bar{m}$, we will call $\lambda$ a weight of $V$, $V(\lambda)_{\bar{m}}$ its (generalized) weight space, and any $v \in V(\lambda)_{\bar{m}}$ a generalized weight vector of degree $\bar{m}$.
\begin{definition}
A $\UHbar{\mfg}$-module $V$ is called a generalized weight module if it splits as a direct sum of generalized weight spaces and $K_j=q_j^{H_j}$ as operators on $V$. We denote by $\mcC$ the category of finite dimensional generalized weight modules.
\end{definition}
Note that if $\gamma =\sum_{j=1}^nc_j\alpha_j \in L$, then $K_{\gamma}=\prod_{j=1}^n q_j^{c_j H_j}$. Notice that for any $c \in \mbbC$, we have 
\[ (H_i-c)X_{\pm j}=(H_iX_{\pm j}-cX_{\pm j})=(\pm a_{ij}X_{\pm j}+X_{\pm j}H_i-cX_{\pm j})=X_{\pm j}(H_i-c \pm a_{ij}).\]
Applying this formula successively shows that
\begin{equation}\label{eq:HX} (H_i-c)^nX_{\pm j}=X_{\pm j}(H_i-c\pm a_{ij})^n,\end{equation}
which has the following lemma as an immediate consequence:
\begin{lemma}\label{Lem:genweight}
Let $V \in \mcC$ and $\lambda \in \mfh^*$ a weight of $V$. Then,
\[ X_{\pm j} V(\lambda)_{\bar{m}} \subset V(\lambda \pm \alpha_j)_{\bar{m}}, \qquad (H_j-\lambda(H_j)) V(\lambda)_{(m_1,...,m_n)}\subset V(\lambda)_{(m_1,...,m_j-1,...,m_n)}. \]
\end{lemma}
We place a partial order on $\mathbb{N}^n$ by setting $\bar{k} \leq \bar{m}$ if $k_j \leq m_j$ for $j=1,...,n$. Let $\mcC_{\bar{m}}$ denote the subcategory of generalized weight modules with weight spaces of degree at most $\bar{m}$ with respect to this partial order. It is clear that $\mcC_{\bar{k}} \subset \mcC_{\bar{m}}$ if $\bar{k} \leq \bar{m}$ and if $K \subset \mathbb{N}^n$ is totally ordered and unbounded with respect to the partial order on $\mathbb{N}^n$, then $\bigcup_{\bar{k}\in K} \mcC_{\bar{k}} = \mcC$. Such a family of subcategories is called a filtration of $\mcC$ if it respects the monoidal structure in the sense that for any $V_1 \in \mcC_{\bar{k_1}},V_2 \in \mcC_{\bar{k_2}}$, we have $V_1 \otimes V_2 \in \mcC_{\bar{k_1}+\bar{k_2}}$.

 Let $V_1,V_2 \in \mcC$ and $v_1 \in V_1(\lambda_1)_{\bar{k}},v_2 \in V_2(\lambda_2)_{\bar{k}'}$. Then for any $n \in \mathbb{N}$ we have
\begin{align*}
(H_i-(\lambda_1+\lambda_2)(H_i))^n \cdot v_1 \otimes v_2&=(\Delta(H) -(\lambda_1+\lambda_2)(H_i))^nv_1 \otimes v_2\\
&=(H_i\otimes 1-\lambda_1(H_i) + 1 \otimes H_i-\lambda_2(H_i))^n v_1 \otimes v_2\\
&=\sum\limits_{m=0}^n \binom{n}{m}(H_iv_1-\lambda_1(H_i))^mv_1 \otimes (H_iv_2-\lambda_2(H_i))^{n-m}v_2.
\end{align*}
The terms of this sum vanish if $m\geq k_i+1$ or $n-m \geq k_i'+1$. If $n=k_i+k_i'+1$, then for $m\leq k_i$ we have $n-m=k_i+k_i'-m+1 \geq k_i'+1$ so the sum vanishes. We therefore have the following proposition:
\begin{proposition}\label{Prop:weighttensor}
Let $V_1,V_2 \in \mcC$ and $v_1 \in V_1(\lambda_1)_{\bar{k_1}}$, $v_2 \in V_2(\lambda_2)_{\bar{k_2}}$. Then,
\[ v_1 \otimes v_2 \in (V_1 \otimes V_2)(\lambda_1+\lambda_2)_{\bar{k_1}+\bar{k_2}}\]
In particular, if $V_1 \in \mcC_{\bar{k_1}},V_2 \in \mcC_{\bar{k_2}}$, then $V_1 \otimes V_2 \in \mcC_{\bar{k_1}+\bar{k_2}}$ so $(\mcC_{\bar{k}})_{\bar{k} \in K}$ is a filtration on $\mcC$.
\end{proposition}
Notice that the usual category of weight modules on which $H,K$ act semisimply is precisely $\mcC_0$. It is easy to see that the irreducible modules in $\mcC$ lie in $\mcC_0$. Indeed, let $V \in \mcC$ be irreducible and let $V_{\mathrm{max}}:=\{v \in V \, | \, X_kv=0 \; \forall i=1,...,N \}$ denote the subset of maximal vectors in $V$. Since $V$ is finite-dimensional, $V_{\mathrm{max}}$ must be non-empty, otherwise we could construct via the action of $X_1,...,X_n$ an infinite chain of independent vectors in $V$. Further, $V_{\mathrm{max}}$ is invariant under the action of $H_k$ with $k=1,...,n$ since if $v \in V_{\mathrm{max}}$, we have
\[ X_jH_iv=H_iX_jv-a_{ij}X_jv=0. \]
Therefore, $H_1,...,H_n$ have a simultaneous eigenvector $\tilde{v} \in V_{\mathrm{max}}$ as they are commuting operators acting on a complex finite-dimenisonal vector space. That is, $\tilde{v} \in V$ is a highest weight vector. Since $V$ is irreducible, it is generated by $\tilde{v}$ and therefore lies in $\mcC_0$ by Lemma \ref{Lem:genweight}:
\begin{proposition}\label{Prop:irred}
If $V \in \mcC$ is irreducible, then $V \in \mcC_0$.
\end{proposition}
The irreducible modules in $\mcC_0$ are the simple quotients of Verma modules in $\mcC_0$ as described in \cite[Proposition 4.3]{R}. 

Given a weight $\lambda \in \mfh^*$ and $\bar{m} \in \mathbb{N}^n$, denote by $\mbbC_{\lambda}^{\bar{m}}$ the $\UHbar{\mfg}^0$-module with vector-space basis $\{v_{\lambda}^{\bar{k}} \; | \; \bar{k}=(k_1,...,k_n) \in \mathbb{N}^n, \, 0 \leq k_j \leq m_j \}$ satisfying $(H_j-\lambda(H_j))v_{\lambda}^{\bar{k}}=0$ if $k_j=0$ and
\begin{equation}\label{eq:Clambda0} (H_j-\lambda(H_j))v_{\lambda}^{\bar{k}}=v_{\lambda}^{k_1,...,k_j-1,...,k_n}, \quad  K_{j}v_{\lambda}^{\bar{k}}=q_j^{\lambda(H_j)}\sum\limits_{s=0}^{k_j} \left(\frac{2 d_j \pi \mathsf{i}}{\ell}\right)^s v_{\lambda}^{k_1,...,k_j-s,...,k_n}
\end{equation}
for $j=1,...,n$. Since $\UHbar{\mfg}^0$ is commutative, it is clear that this is a well-defined $\UHbar{\mfg}^0$-module. Further, we have $K_j=q_j^{H_j}$ as operators on $\mathbb{C}_{\lambda}^{\bar{m}}$ since
\begin{align*}
q_j^{H_j}v_{\lambda}^{\bar{k}}&=q_j^{H_j-\lambda(H_j)+\lambda(H_j)}v_{\lambda}^{\bar{k}}\\
&=q_j^{\lambda(H_j)}\sum\limits_{s=0}^{\infty} \left(\frac{2d_j \pi \mathsf{i} }{\ell}\right)^s (H_j-\lambda(H_j))^s v_{\lambda}^{\bar{k}}\\
&=q_j^{\lambda(H_j)}\sum\limits_{s=0}^{k_j-1} \left(\frac{2d_j \pi \mathsf{i} }{\ell}\right)^s (H_j-\lambda(H_j))^s v_{\lambda}^{\bar{k}}\\
&=q_j^{\lambda(H_j)}\sum\limits_{s=0}^{k_j-1} \left(\frac{2d_j \pi \mathsf{i} }{\ell}\right)^s v_{\lambda}^{k_1,...,k_j-s,...,k_n}\\
&=K_jv_{\lambda}^{\bar{k}}.\\
\end{align*}
We can now extend $\mbbC_{\lambda}^{\bar{m}}$ to a $\overline{U}_q^H(\mfg)^+$-module by letting $X_j,j=1,...,n$ act as zero. 
\begin{definition}\label{Def:genverma}
We define the generalized Verma module of highest weight $\lambda \in \mfh^*$ and degree $\bar{m} \in \mathbb{N}^n$ by  
\[ \GVerma{\lambda}{\bar{m}}:=\UHbar{\mfg} \otimes_{\overline{U}_q^H(\mfg)^+} \mbbC_{\lambda}^{\bar{m}}.\]
 $\GVerma{\lambda}{\bar{m}}$ is generated by the vector $v_{\lambda}^{\bar{m}}$ satisfying Equation \eqref{eq:Clambda0} and $X_jv_{\lambda}^{\bar{m}}=0$ for $j=1,...,n$.
\end{definition}

It is clear that $\GVerma{\lambda}{\bar{m}} \in \mcC_{\bar{m}}$ and it is universal with respect to generalized highest weight modules of degree $\bar{m}$ as it is a free $\overline{U}_q^H(\mfg)^-$-module. That is, for any module $M$ generated by a highest weight vector of degree $\bar{m}$ and weight $\lambda \in \mfh^*$, there is a surjection $\GVerma{\lambda}{\bar{m}} \twoheadrightarrow M$. In particular, every irreducible module in $\mcC$ is realized as a quotient of $\GVerma{\lambda}{\bar{0}}$ and we will denote the irreducible quotient of $\GVerma{\lambda}{\bar{0}}$ by $\irred{\lambda}$. It follows from the definition of $\GVerma{\lambda}{\bar{m}}$ and Equation \eqref{eq:basis} that
\begin{equation}\label{eq:Vermabasis} 
\{X_{\beta_1}^{t_1}X_{\beta_2}^{t_2} \cdots X_{\beta_N}^{t_N}v_{\lambda}^{\bar{k}}\; |\; 0 \leq t_i \leq r_{\beta_i}-1, \; 0 \leq k_j \leq m_j \}
\end{equation}
is a $\mathbb{C}$-basis for $\GVerma{\lambda}{\bar{m}}$. We now introduce a few definitions necessary for the following subsection. First, we introduce the duality functor $M \mapsto \check{M}$. We refer the reader to \cite[Subsection 4.2]{R} for a more detailed treatment of this functor and its properties. 
\begin{definition}\label{Def:dual}
Let $M \in \mcC$ and let $M^* \in \mcC$ denote the usual module structure on $\Hom_{\mathbb{C}}(M,\mathbb{C})$ defined by the antipode $S$ of $\UHbar{\mfg}$. We define $\check{M}$ to be the module obtained by twisting the action of $\UHbar{\mfg}$ on the  $M^*$ of $M$ by the automorphism $\omega$ defined in Equation \eqref{eq:aut}.\end{definition}
That is, we let $x \in \UHbar{\mfg}$ act on $M^*$ by $\omega(x)$. $\check{M}$ can also be understood as the dual defined by $S \circ \omega$. It is easy to see that generalized weight modules are preserved under duality so $M \mapsto \check{M}$ defines an endofunctor on $\mcC$. Some useful properties of this functor are summarized in the following lemma:
\begin{lemma}\label{Lem:duality} For all $M \in \mcC$ and $\lambda \in \mfh^*$, we have the following:
\begin{itemize}
\item The duality functor $M \mapsto \check{M}$ is contravariant, exact, and involutive $\check{\check{M}} \cong M$.
\item $\dim M(\lambda)=\dim \check{M}(\lambda)$.
\item $\check{\irred{\lambda}} \cong \irred{\lambda}$ .
\end{itemize}
\end{lemma}
We now introduce the notion of standard and costandard filtrations of fixed degree which will be fundamental for the proof of BGG reciprocity in Theorem \ref{Thm:BGG}. In what follows a dual generalized Verma module will always refer to the image $\DGVerma{\lambda}{\bar{k}}$ of a generalized Verma module $\GVerma{\lambda}{\bar{k}}$ under the duality functor.
\begin{definition}\label{Def:standard}
We call a filtration of $M \in \mcC$ standard of degree $\bar{k}$ if its successive quotients are generalized Verma modules of degree $\bar{k}$ and costandard if its successive quotients are dual Verma modules of degree $\bar{k}$.
\end{definition}
To be precise, a filtration $0=M_0 \subset M_1 \subset \hdots M_k=M$ of some $M \in \mcC$ is standard of degree $\bar{k} \in \mathbb{N}^n$ if $M_j/M_{j-1} \cong \GVerma{\lambda_j}{\bar{k}}$ for some $\lambda_j \in \mfh^*$. Similarly, we call such a filtration costandard of degree $M$ if $M_j/M_{j-1} \cong \check{M}_{\lambda_j}^{\bar{k}}$ for some $\lambda_j \in \mfh^*$.

\subsection{Projective Modules}\label{subsec:proj}
For any $\alpha \in \Delta^+$ and $\lambda \in \mfh^*$, we associate the scalar $\lambda_{\alpha} \in \mbbC$ defined by
\[ \lambda_{\alpha}:= \langle \lambda+ \rho , \alpha \rangle\]
where $\rho$ denotes the Weyl vector. We call $\lambda_{\alpha}$ typical iff $\lambda_{\alpha} \in \ddot{\mbbC}_{\alpha}$ where
\[ \ddot{\mbbC}_{\alpha}:=\begin{cases}
(\mbbC \setminus g_{\alpha} \mathbb{Z}) \cup r\mathbb{Z} & \text{ if $\ell$ is even}\\
(\mbbC \setminus \frac{g_{\alpha}}{2}\mathbb{Z}) \cup \frac{r}{2}\mathbb{Z} & \text{ if $\ell$ is odd}.
\end{cases}\]
where $g_{\alpha}=\mathrm{gcd}(d_{\alpha},r)$. 
\begin{definition}\label{Def:typical}
We call $\lambda \in \mfh^*$ typical if $\lambda_{\alpha}$ is typical for all $\alpha \in \Delta^+$ and atypical otherwise.
\end{definition}
The Verma modules $\GVerma{\lambda}{\bar{0}} \in \mcC_{\bar{0}}$ are both irreducible and projective iff $\lambda$ is typical (\cite[Section 4]{R} \& \cite[Lemma 7.1]{CGP2}. Although the generalized Verma modules are never irreducible (they clearly always have $\GVerma{\lambda}{\bar{0}}$ as a proper submodule), their projectivity is still determined by the typicality of $\lambda$.
\begin{lemma}\label{Lem:Vermaproj}
The generalized Verma module $\GVerma{\lambda}{\bar{m}}$ is projective in $\mcC_{\bar{m}}$ if $\lambda$ is typical.
\end{lemma}
\begin{proof}
Let 
\[ X_+:=\prod\limits_{k=1}^N X_{\beta_k}^{r_{\beta_k}-1}, \qquad \mathrm{and} \qquad X_-:=\prod\limits_{k=1}^NX_{\beta_k}^{r_{\beta_k}-1} \]
denote the highest and lowest weight vectors in $\UHbar{\mfg}$ respectively and $N \subset \GVerma{\lambda}{\bar{m}}$ the unique maximal submodule such that $\GVerma{\lambda}{\bar{m}}/N \cong \irred{\lambda}$ (N consists of the vectors of degree less than $\bar{m}$). If $\lambda$ is typical then $\irred{\lambda} \cong \GVerma{\lambda}{\bar{0}}$ so $X_+X_-v_{\lambda}^{\bar{m}} \not = 0$ as its image is non-zero in $\GVerma{\lambda}{\bar{m}}/N \cong \irred{\lambda} \cong \GVerma{\lambda}{\bar{0}}$. For any $\bar{k} \in \mathbb{N}$ we have $X_+X_-v_{\lambda}^{\bar{k}} \in \GVerma{\lambda}{\bar{m}}(\lambda)_{\bar{k}}$ by Lemma \ref{Lem:genweight} so it follows from the basis for $\GVerma{\lambda}{\bar{m}}$ given in Equation \eqref{eq:Vermabasis} and the module structure of $\mathbb{C}_{\lambda}^{\bar{m}}$ that
\begin{equation}\label{eq:+-} X_+X_-v_{\lambda}^{\bar{k}}=\sum\limits_{\bar{k}'\leq \bar{k}} \nu_{\bar{k}'}^{\bar{k}}v_{\lambda}^{\bar{k}'} \end{equation}
for some $\nu_{\bar{k}'}^{\bar{k}} \in \mathbb{C}$. Notice that $\nu_{\bar{k}}^{\bar{k}} \not = 0$ since $\nu_{\bar{k}}^{\bar{k}}v_{\lambda}^{\bar{k}}+N^{\bar{k}}=X_+X_-v_{\lambda}^{\bar{k}}+N^{\bar{k}} \not = 0$ in $\GVerma{\lambda}{\bar{k}}/N^{\bar{k}} \cong \irred{\lambda}$ where we now identify $\GVerma{\lambda}{\bar{k}}$ with the subalgebra of $\GVerma{\lambda}{\bar{m}}$ generated by $v_{\lambda}^{\bar{k}}$ and $N^{\bar{k}}$ is the unique maximal submodule of $\GVerma{\lambda}{\bar{k}}$. Define the scalars $\gamma_{\bar{k}}$ for $\bar{k} < \bar{m}$ inductively by the relation
\[ \gamma_{\bar{k}}:= \frac{1}{\nu_{\bar{k}}^{\bar{k}}} \left(\frac{\nu_{\bar{k}}^{\bar{m}}}{\nu_{\bar{m}}^{\bar{m}}} - \sum\limits_{\bar{k}< \bar{k}' < \bar{m}} \gamma_{\bar{k}'} \nu_{\bar{k}}^{\bar{k}'} \right) \qquad \Leftrightarrow \qquad   \sum\limits_{\bar{k} \leq \bar{k}' < \bar{m}} \gamma_{\bar{k}'}\nu_{\bar{k}}^{\bar{k}'}=\frac{\nu_{\bar{k}}^{\bar{m}}}{\nu_{\bar{m}}^{\bar{m}}} \]
which is well defined since $\nu_{\bar{k}}^{\bar{k}} \not = 0$ for all $\bar{k} \leq \bar{m}$. If $f:M \to \GVerma{\lambda}{\bar{m}}$ is a surjection, then there exist vectors $u_{\bar{m}} \in f^{-1}( \frac{1}{\nu_{\bar{m}}^{\bar{m}}}v_{\lambda}^{\bar{m}})$ and $u_{\bar{k}} \in f^{-1}(\gamma_{\bar{k}} v_{\lambda}^{\bar{k}})$ for $\bar{k}<\bar{m}$. Let $w=X_+X_-(u_{\bar{m}}-\sum_{\bar{k}< \bar{m}}u_{\bar{k}})$, then by Equation \eqref{eq:+-} we have
\begin{align*}
f(w)&=X_+X_-\left(\frac{1}{\nu_{\bar{m}}^{\bar{m}}} v_{\lambda}^{\bar{m}}- \sum\limits_{\bar{k}<\bar{m}} \gamma_{\bar{k}}v_{\lambda}^{\bar{k}} \right)\\
&= \sum\limits_{\bar{k} \leq \bar{m}} \frac{\nu_{\bar{k}}^{\bar{m}}}{\nu_{\bar{m}}^{\bar{m}}} v_{\lambda}^{\bar{m}} - \sum\limits_{\bar{k}<\bar{m}} \gamma_{\bar{k}} \left( \sum\limits_{\bar{k}' \leq \bar{k}} \nu_{\bar{k}'}^{\bar{k}}v_{\lambda}^{\bar{k}'}\right)\\
&=\sum\limits_{\bar{k} \leq \bar{m}} \frac{\nu_{\bar{k}}^{\bar{m}}}{\nu_{\bar{m}}^{\bar{m}}} v_{\lambda}^{\bar{m}} - \sum\limits_{\bar{k}<\bar{m}}  \left( \sum\limits_{\bar{k} \leq \bar{k}' < \bar{m}} \gamma_{\bar{k}'}\nu_{\bar{k}}^{\bar{k}'}\right)v_{\lambda}^{\bar{k}}\\
&=\sum\limits_{\bar{k} \leq \bar{m}} \frac{\nu_{\bar{k}}^{\bar{m}}}{\nu_{\bar{m}}^{\bar{m}}} v_{\lambda}^{\bar{m}} - \sum\limits_{\bar{k}<\bar{m}} \frac{\nu_{\bar{k}}^{\bar{m}}}{\nu_{\bar{m}}^{\bar{m}}}v_{\lambda}^{\bar{k}}=v_{\lambda}^{\bar{m}}.
\end{align*}
Notice that $X_jX_+=0$ in $\overline{U}_q^H(\mfg)$ and $(H_j-\lambda(H_j))^{m_j}w \not =0$ for $j=1,...,n$ since 
\[ f((H_j-\lambda(H_j))^{m_j}w)=(H_j-\lambda(H_j))^{m_j}v_{\lambda}^{\bar{m}} \not = 0 \]
so $w$ is a generalized highest weight vector of weight $\lambda$ and degree $\bar{m}$. Therefore, by the universal property of generalized Verma modules, there exists a map $g:\GVerma{\lambda}{\bar{m}} \to M$ defined by $g(v_{\lambda}^{\bar{m}})=w$ and $f \circ g=\mathrm{Id}_{\GVerma{\lambda}{\bar{m}}}$ so $f:M \to \GVerma{\lambda}{\bar{m}}$ splits and $\GVerma{\lambda}{\bar{m}}$ is projective.
\end{proof}
The statement of Lemma \ref{Lem:Vermaproj} can be strengthened to an if and only if statement, that is, $\GVerma{\lambda}{\bar{m}}$ is projective in $\mcC_{\bar{m}}$ if and only if $\lambda$ is typical. This is a trivial consequence of BGG reciprocity proved in Theorem \ref{Thm:BGG}.

Given a Hopf algebra $H$ with Hopf subalgebra $S$ and $M,N$ modules over $H$ and $S$ respectively, then
\[ (H \otimes_S N) \otimes M \cong H \otimes_S (N \otimes M|_S)\]
where $M|_S$ is $M$ viewed as an $S$-module (see \cite[Remark 2.6]{ST}). Indeed, the isomorphism is given by
\[ (h \otimes_S n) \otimes m \mapsto \sum h_{(1)} \otimes_S(n \otimes S(h_{(2)})m)\]
with inverse 
\[ h \otimes_S(n \otimes m) \mapsto \sum (h_{(1)} \otimes_S n) \otimes h_{(2)}m.\]
Applying this equivalence to $H=\UHbar{\mfg}$, $S=\overline{U}_q^H(\mfg)^+$ with $N=\mathbb{C}_{\lambda}^{\bar{m}}$, we see that
\begin{equation}\label{eq:Vermatensor} \GVerma{\lambda}{\bar{m}} \otimes M \cong (\UHbar{\mfg}\otimes_{\overline{U}_q^H(\mfg)^+} \mathbb{C}_{\lambda}^{\bar{m}}) \otimes M \cong \UHbar{\mfg}\otimes_{\overline{U}_q^H(\mfg)^+} (\mathbb{C}_{\lambda}^{\bar{m}} \otimes M|_{\overline{U}_q^H(\mfg)^+})\end{equation}
for any $M \in \mcC$. Equation \ref{eq:Vermatensor} is fundamental in establishing the following adaptation of \cite[Theorem 3.6]{H}
\begin{lemma}\label{Lem:tensorstandard}
For any $\lambda \in \mfh^*$, $\bar{m} \in \mathbb{N}^n$, and $M \in \mcC_0$, the tensor product $\GVerma{\lambda}{\bar{m}} \otimes M$ admits a degree $\bar{m}$ standard filtration 
\[ 0 \subset M_k \subset \hdots M_1 =\GVerma{\lambda}{\bar{m}} \otimes M\]
where $M_j/M_{j+1} \cong \GVerma{\lambda+\mu_j}{\bar{m}}$ with $\mu_j$ taking values in the weights of $M$ appearing $\dim M(\mu_j)$ times and $\mu_j \leq \mu_{j+1}$. 
\end{lemma}
\begin{proof}
Let $v_1,...,v_k$ denote an ordered weight basis for $M|_{\overline{U}_q^H(\mfg)^+}$ with weights $\mu_1,...,\mu_k$ and degrees $d_1,...,d_k$ respectively (allowing for repetition) where $i \leq j$ when $\mu_i \leq \mu_j$. For each $1 \leq j \leq k$ let $N_j$ denote the $\overline{U}^H_q(\mfg)^+$-submodule of $N:=\mathbb{C}_{\lambda}^{\bar{m}} \otimes M|_{\overline{U}_q^H(\mfb)}$ spanned by $\{\mathbb{C}_{\lambda}^{\bar{m}} \otimes v_s \, | \, j \leq s \leq k\}$. Then it is clear that we have a filtration 
\[ 0 \subset  N_k \subset N_{k-1} \subset \cdots N_1=N. \]
Since $X_sv_j=0$ for all $s=1,...,n$, it is easy to see that $\mathbb{C}_{\lambda}^{\bar{m}} \otimes v_j \cong \mathbb{C}_{\lambda+\mu_j}^{\bar{m}}$ as $\overline{U}_q^H(\mfg)^+$-modules by the proof of Proposition \ref{Prop:weighttensor} (here we require $M \in \mcC_0$) so we have $N_j/N_{j+1} \cong \mathbb{C}_{\lambda+\mu_j}^{\bar{m}}$. Since the functor
\[ \UHbar{\mfg}\otimes_{\overline{U}_q^H(\mfb)} (\mathbb{C}_{\lambda}^{\bar{m}} \otimes  (-))\]
is exact as it is a composition of exact functors (notice that $\UHbar{\mfg}$ is pointed, hence flat over its Hopf subalgebras), it follows that this filtration induces to a filtration
\[ M_j:=\UHbar{\mfg}\otimes_{\overline{U}_q^H(\mfb)} (\mathbb{C}_{\lambda}^{\bar{m}} \otimes N_j)\]
of $\UHbar{\mfg}\otimes_{\overline{U}_q^H(\mfb)} N \cong \GVerma{\lambda}{\bar{m}} \otimes M$ (by Equation \eqref{eq:Vermatensor}) satisfying $M_j/M_{j+1} \cong \GVerma{\lambda+\mu_j}{\bar{m}}$.
\end{proof}
We require one additional lemma before we can establish results on projective covers.
\begin{lemma}\label{Lem:tensorproj}
Let $P \in \mcC_{\bar{m}}$ be projective. Then for any $M \in \mcC_{\bar{0}}$, $P \otimes M$ and $M \otimes P$ are projective in $\mcC_{\bar{m}}$.
\end{lemma}
The lemma follows in the usual way from the fact that
\[ \mathrm{Hom}_{\mcC_{\bar{m}}}(P \otimes M,-) \cong \mathrm{Hom}_{\mcC_{\bar{m}}}(P, (- \otimes M^*)) \cong \mathrm{Hom}_{\mcC_{\bar{m}}}(P,-) \circ (- \otimes M^*)\]
is a composition of exact functors, where we need only observe that $P \otimes M \in \mcC_{\bar{m}}$ by Lemma \ref{Lem:genweight}.

We now have the tools required to prove the existence of projective covers. 
\begin{theorem}\label{Thm:projexist}
\begin{itemize}
\item Projective modules do not exist in $\mcC$.
\item For any $\bar{m} \in \mathbb{N}^n$ and $\lambda \in \mfh^*$, $\irred{\lambda}$ admits a projective cover in $\mcC_{\bar{m}}$.
\end{itemize}
\end{theorem}
\begin{proof}
Suppose $P \in \mcC$ is projective, let $\mathrm{Rad}(P)$ denote the union of its maximal submodules, and $\irred{\lambda}$ any irreducible factor of $P/\mathrm{Rad}(P)$. Then $P$ admits a surjection onto $\irred{\lambda}$. Since $P$ is finite-dimensional, its weight spaces have finite degree so there exists some $\bar{m} \in \mathbb{N}^n$ such that $P \in \mcC_{\bar{m}}$. Let $\bar{m}'>\bar{m}$, then there exists a surjection $f:\GVerma{\lambda}{\bar{m}'} \to \irred{\lambda}$. Since $P$ is projective, there is a map $h:P \to \GVerma{\lambda}{\bar{m}'}$ such that the following diagram commutes:\\
\vspace{-0.5cm}
\begin{center}
\begin{tikzcd}
\GVerma{\lambda}{\bar{m}'}\arrow[twoheadrightarrow]{dr}{f} & \\
P\arrow{u}{h} \arrow[twoheadrightarrow]{r}{p}&  \irred{\lambda}
\end{tikzcd}
\end{center}
It follows from the definition of $f:\GVerma{\lambda}{\bar{m}'} \to \irred{\lambda}$ that $h:P \to \GVerma{\lambda}{\bar{m}'}$ must be surjective, else $f \circ h=0$, so there exists some $u \in P$ such that $h(u)=v_{\lambda}^{\bar{m}'}$ where $v_{\lambda}^{\bar{m}'}\in \GVerma{\lambda}{\bar{m}'}(\lambda)_{\bar{m}'}$ is non-zero. Notice now that 
\[ h((H_j-\lambda(H_j))^{\bar{m}'}u)=(H_j-\lambda(H_j))^{\bar{m}'}h(u)=(H_j-\lambda(H_j))^{\bar{m}'} v_{\lambda}^{\bar{m}'} \not = 0\]
so $(H_j-\lambda(H_j))^{\bar{m}'}u\not =0$ for $j=1,...,n$, contradicting $P \in \mcC_{\bar{m}}$. Therefore, a projective module cannot exist in $\mcC$.

By Lemma \ref{Lem:Vermaproj}, $\GVerma{\mu}{\bar{m}} \in \mcC_{\bar{m}}$ is projective for typical $\mu$. Hence, for any $\lambda \in \mfh^*$, choose $\mu \in \mfh^*$ such that $\lambda-\mu$ is typical. Then the module
\[ \GVerma{\lambda-\mu}{\bar{m}}\otimes \irred{-\mu}^*\]
is projective by Lemmas \ref{Lem:Vermaproj} \& \ref{Lem:tensorproj}. The lowest weight of $\irred{-\mu}^*$ is $\mu$ so $\GVerma{\lambda}{\bar{m}}$ is quotient of $\GVerma{\lambda-\mu}{\bar{m}}\otimes \irred{-\mu}^*$ by Lemma \ref{Lem:tensorstandard}. Hence, $\irred{\lambda}$ is a quotient of $\GVerma{\lambda-\mu}{\bar{m}}\otimes \irred{-\mu}^*$ and since $\GVerma{\lambda-\mu}{\bar{m}}\otimes \irred{-\mu}^*$ is finite-dimensional, it follows from \cite[Lemma 3.6]{Kr} that its indecomposable summand which surjects onto $\irred{\lambda}$ is the projective cover of $\irred{\lambda}$. 
\end{proof}
We will now denote by $\Proj{\lambda}{\bar{m}}$ the projective cover of $\irred{\lambda}$ in $\mcC_{\bar{m}}$. Notice that if $\bar{k} < \bar{m}$, then $\Proj{\lambda}{\bar{k}} \in \mcC_{\bar{m}}$, although it is not projective in $\mcC_{\bar{m}}$. We have the following:
\begin{proposition}
Let $\lambda \in \mfh^*$ and $\bar{k}<\bar{m}$ both in $\mathbb{N}^n$. Then $\Proj{\lambda}{\bar{m}}$ is the projective cover of $\Proj{\lambda}{\bar{k}}$ viewed as objects in $\mcC_{\bar{m}}$. 
\end{proposition}
\begin{proof}
Both $\Proj{\lambda}{\bar{m}}$ and $\Proj{\lambda}{\bar{k}}$ admit essential surjections
\[ \varphi:\Proj{\lambda}{\bar{m}} \to \irred{\lambda}, \qquad \phi:\Proj{\lambda}{\bar{k}} \to \irred{\lambda}.\]
Since $\Proj{\lambda}{\bar{m}}$ is projective in $\mcC_{\bar{m}}$, and $\irred{\lambda},\Proj{\lambda}{\bar{k}} \in \mcC_{\bar{m}}$, there exists a morphism $f:\Proj{\lambda}{\bar{m}} \to \Proj{\lambda}{\bar{k}}$ satisfying $\phi \circ f=\varphi$. Notice that if $f$ is not surjective, $\mathrm{Im}(f)$ would be a proper submodule such that $\varphi(\Proj{\lambda}{\bar{k}})=\irred{\lambda}=\phi(\mathrm{Im}(f))$ contradicting the fact that $\phi:\Proj{\lambda}{\bar{k}} \to \irred{\lambda}$ is essential. Therefore, $f:\Proj{\lambda}{\bar{m}} \to \Proj{\lambda}{\bar{k}}$ is surjective and we need only show it is essential. Suppose that $\Proj{\lambda}{\bar{m}}$ admits a proper submodule $N$ such that $f(N)=\Proj{\lambda}{\bar{k}}$. It then follows that 
\[ \varphi(N)=\phi(f(N))=\phi(P_{\lambda}^{\bar{k}})=\irred{\lambda},\]
which contradicts the fact that $\varphi$ is essential. Hence no such submodule exists, so $f:\Proj{\lambda}{\bar{m}} \to \Proj{\lambda}{\bar{k}}$ is essential.
\end{proof}
In the proof of Theorem \ref{Thm:projexist}, $\Proj{\lambda}{\bar{m}}$ is realized as an indecomposable summand of a projective module of the form $\GVerma{\lambda-\mu}{\bar{m}}\otimes \irred{-\mu}^*$. An argument identical to \cite[Proposition 3.7]{H} shows that summands of modules with degree $\bar{m}$ standard filtrations admit degree $\bar{m}$ standard filtrations themselves so we have the following consequence of Lemma \ref{Lem:tensorstandard} since all projectives are direct sums of projective covers.
\begin{lemma}\label{Prop:projfilt}
Any projective object $P \in \mcC_{\bar{m}}$ admits a degree $\bar{m}$ standard filtration.
\end{lemma}
We have seen that for each $\bar{m} \in \mbbN^n$, $\mcC_{\bar{m}}$ has enough projectives. We now demonstrate that a variant of BGG reciprocity holds in these categories as well. Recall that for any $M \in \mcC$, $[M:\irred{\lambda}]$ denotes the multiplicity of $\irred{\lambda}$ in the Jordan-Holder series of $M$. Similarly, if $M$ admits a degree $\bar{m}$ standard filtration, we denote the multiplicity of $\GVerma{\lambda}{\bar{m}}$ in $M$ by $(M,\GVerma{\lambda}{\bar{m}})$. Then we have the following variant of BGG reciprocity:
\begin{theorem}[BGG Reciprocity] \label{Thm:BGG}
For all $\lambda,\mu \in \mfh^*$ and $\bar{m} \in \mathbb{N}^n$ we have
\[ (\Proj{\lambda}{\bar{m}},\GVerma{\mu}{\bar{m}})=[\Verma{\mu},\irred{\lambda}].\]
\end{theorem}
\begin{proof}
The projective cover $\Proj{\lambda}{\bar{m}} \in \mcC_{\bar{m}}$ of $\irred{\lambda}$ has a degree $\bar{m}$ standard filtration by Proposition \ref{Prop:projfilt}. Therefore, if for any $M \in \mcC$ we have
\begin{enumerate}
\item[$(1)$] $\dim \Hom(P_{\lambda}^{\bar{m}},M)=[M,\irred{\lambda}]$, and
\item[$(2)$] if $M$ admits a degree $\bar{m}$ standard filtration, $(M,\GVerma{\lambda}{\bar{m}})=\dim \Hom (M,\DVerma{\lambda})$.
\end{enumerate}
Then,
\[ (\Proj{\lambda}{\bar{m}},\GVerma{\mu}{\bar{m}})=\dim \Hom(\Proj{\lambda}{\bar{m}},\DVerma{\mu})=[\DVerma{\mu},\irred{\lambda}]=[\Verma{\mu},\irred{\lambda}]\]
where we have used the fact that if $\irred{\lambda}$ is a factor of $\DVerma{\mu}$, then $\irred{\lambda} \cong \check{\irred{\lambda}}$ is a factor of $\Verma{\mu}$. The argument for $(1)$ is precisely the argument of \cite[Theorem 3.9 (c)]{H}, so we prove only the second claim. First, notice that $\gamma \leq \mu$ for all weights $\gamma$ of $\DVerma{\mu}$ since $\dim \DVerma{\mu}(\gamma)= \Verma{\mu}(\gamma)$ (recall Lemma \ref{Lem:duality}) so if there exists a morphism $\phi:\GVerma{\mu}{\bar{m}} \to \DVerma{\lambda}$ we must have $\mu \leq \lambda$. Further, $\irred{\lambda}$ is the unique irreducible submodule of $\DVerma{\lambda}$ so it must appear as a factor in $\mathrm{Im}(\phi)$, hence $\lambda \leq \mu$. It is also easy to see that $\dim \Hom(\GVerma{\lambda}{\bar{m}},\DVerma{\lambda})=1$ so
\[ \dim \Hom(\GVerma{\mu}{\bar{m}},\DVerma{\lambda})=\delta_{\lambda,\mu}.\]
We now prove the second claim by induction on the filtration length of $M$ with the length $1$ case following from the above equation and $(\GVerma{\mu}{\bar{m}},\GVerma{\lambda}{\bar{m}})=\delta_{\lambda,\mu}$. Suppose now that the filtration length of $M$ is $k$ and that the statement holds for modules whose filtration length is $k-1$. Then there exists a submodule $N$ of $M$ whose degree $\bar{k}$ standard filtration has length $k-1$ and satisfies a short exact sequence
\[ 0 \to N \to M \to \GVerma{\mu}{\bar{m}} \to 0.\]
 We have an induced long exact sequence
\begin{align*}
0 \to \Hom(\GVerma{\mu}{\bar{m}},\DVerma{\lambda}) &\to \Hom(M, \DVerma{\lambda}) \to \Hom(N,\DVerma{\lambda}) \to\\
& \to \mathrm{Ext}(\GVerma{\mu}{\bar{m}},\DVerma{\lambda}) \to \hdots 
\end{align*}
$\GVerma{\mu}{\bar{m}}$ clearly admits a degree $\bar{0}$ standard filtration, so $\mathrm{Ext}(\GVerma{\mu}{\bar{m}},\DVerma{\lambda})=0$ by Lemma \ref{Lem:Ext}. Hence, we have 
\begin{align*}
\dim \Hom(M,\DVerma{\lambda})&= \dim \Hom(N,\DVerma{\lambda})+\dim \Hom(\GVerma{\mu}{\bar{m}},\Verma{\lambda})\\
&=(N,\GVerma{\lambda}{\bar{m}})+(\GVerma{\mu}{\bar{m}},\Verma{\lambda})\\
&=(M,\GVerma{\lambda}{\bar{m}}).
\end{align*}
\end{proof}

We finish this subsection with an observation that the projective modules in $\mcC_{\bar{m}}$ can be characterized as tilting modules. This result will be necessary in Section \ref{Sec:projcoversl2} to give a generators and relations description for projective covers when $\mfg=\mathfrak{sl}_2$.
\begin{definition}
We call $M \in \mcC_{\bar{m}}$ a tilting module if it admits both a standard and costandard filtration of degree $\bar{m}$ (recall Definition \ref{Def:standard}).
\end{definition}
The argument for \cite[Theorem 4.15]{R} can be applied verbatim to show that projective modules are self-dual under the duality functor $M \mapsto \check{M}$, hence they are tilting since they have degree $\bar{m}$ standard filtrations which induce to degree $\bar{m}$ costandard filtrations under duality. It is known that for finite-dimensional algebras with triangular decomposition, tilting modules are precisely the projective-injective objects in the category of graded finite-dimensional modules \cite[Theorem 5.1]{BT}. The unrolled quantum group is not finite-dimensional, but we can none the less adapt the proof to our setting.  If $M \in \mcC_{\bar{m}}$ is tilting, then in particular it admits a degree $\bar{m}$ costandard filtration
\[ 0=M_0 \subset M_1 \subset \hdots M_k=M\]
with $M_j/M_{j-1} \cong \DGVerma{\lambda_j}{\bar{m}}$ for some $\lambda_j \in \mfh^*$. We therefore have a surjection $q:M \twoheadrightarrow \DGVerma{\lambda_k}{\bar{m}}$. Dual generalized Verma modules are uniquely characterized by their degree and unique simple quotient (as generalized Verma modules are also characterized by their degree and unique simple submodule). Therefore, if $\irred{\lambda}$ is the unique simple quotient of $\DGVerma{\lambda_k}{\bar{m}}$, $\DGVerma{\lambda_k}{\bar{m}}$ appears as the quotient of the costandard filtration of $\Proj{\lambda}{\bar{m}}$. We therefore have a surjection $p:\Proj{\lambda}{\bar{m}} \twoheadrightarrow \DGVerma{\lambda_k}{\bar{m}}$ and a map $h:\Proj{\lambda}{\bar{m}} \to M$ such that $p=q \circ h$ since $\Proj{\lambda}{\bar{m}}$ is projective in $\mcC_{\bar{m}}$. Consider now the short exact sequence
\[ 0 \to \mathrm{Ker}(p) \to \Proj{\lambda}{\bar{m}} \stackrel{p}{\to} \DGVerma{\lambda_k}{\bar{m}} \to 0.\]
Since $\DGVerma{\lambda_k}{\bar{m}}$ is the top of the degree $\bar{m}$ costandard filtration of $\Proj{\lambda}{\bar{m}}$, $\mathrm{Ker}(p)$ also admits a degree $\bar{m}$ costandard filtration. As $M$ is tilting, it admits a degree $\bar{m}$ standard filtration as well so $\mathrm{Ext}(M,\mathrm{Ker}(p))=0$ by the following useful lemma, which can be shown by a standard argument (see \cite[Section 6.12]{H}).
\begin{lemma}\label{Lem:Ext}
If $M,N$ have standard filtrations of the same degree, then $\mathrm{Ext}^n(M,\check{N})=0$ for all $n>0$.
\end{lemma}
We therefore obtain an exact sequence
\[ 0 \to \Hom(M,\Ker(p)) \to \Hom(M,\Proj{\lambda}{\bar{m}}) \stackrel{p_*}{\to} \Hom(M,\DGVerma{\lambda_k}{\bar{m}}) \to 0\]
where $p_*$ denotes composition by $p$. In particular,
\[ p_*:\Hom(M,\Proj{\lambda}{\bar{m}}) \to \Hom(M,\DGVerma{\lambda_k}{\bar{m}})\]
is a surjection, so there exists some $g \in \Hom(M,\Proj{\lambda}{\bar{m}})$ such that $p \circ g=q$. It follows now that
\[p \circ g \circ h = q \circ h=p \]
so $g \circ h:\Proj{\lambda}{\bar{m}}\to \Proj{\lambda}{\bar{m}}$ is an isomorphism since $p$ is essential and therefore $g:M \to \Proj{\lambda}{\bar{m}}$ is surjective. By projectivity, $\Proj{\lambda}{\bar{m}}$ is a summand of $M$ so it follows by a trivial induction argument on filtration length that $M$ is in fact projective. We therefore have the following lemma.
\begin{proposition}\label{Prop:tiltingproj}
For any $\bar{m} \in \mathbb{N}^n$, $M \in \mcC_{\bar{m}}$ is tilting if and only if it is projective. 
\end{proposition}

\section{Projective Covers for Irreducible $\UHbar{\mathfrak{sl}_2}$-modules.}\label{Sec:projcoversl2}
We now restrict ourselves to the rank one case $\overline{U}_q^H(\mathfrak{sl}_2)$. To maintain consistency with existing literature on unrolled quantum groups, we will adopt the following change of notation: $\overline{U}_q^H(\mathfrak{sl}_2)$ is the $\mathbb{C}$-algebra generated by $K,K^{-1},E,F,H$ subject to
\begin{align}\label{eq:sl21}
 KK^{-1}&=K^{-1}K=1, & KE&=q^2EK, & KF&=q^{-2}FK, & [E,F]&=\frac{K-K^{-1}}{q-q^{-1}},\\
\label{eq:sl22}HK&=KH, & [H,E]&=2E, & [H,F]&=-2F.
\end{align}
In what follows, recall that $r=\ell$ if $\ell$ is odd and $r=\ell/2$ if $\ell$ is even. By Proposition \ref{Prop:irred}, the irreducible modules in $\mcC$ are the finite-dimensional highest weight modules. These were classified for $\overline{U}_q^H(\mathfrak{sl}_2)$ in \cite{CGP}. For any $\alpha \in \mathbb{C}$, let $M_{\alpha}$ be the $r$-dimensional module generated by a highest weight vector of weight $\alpha$ (i.e. the Verma module of highest weight $\alpha$). For each $i \in \{0,...,r-2\}$ let $L_i$ denote the $(i+1)$-dimensional irreducible module with vector space basis $\{s_0,s_1,...,s_i\}$ satisfying
\[ Fs_k=s_{k+1}, \quad Hs_k=(i-2k)s_k, \quad Es_0=Fs_i=0, \quad Es_k=[k][i+1-k]s_{k-1}. \]
Further, let $\mathbb{C}^H_{\frac{k\ell}{2}}$ be the one dimensional module on which $E,F$ act as zero and $H$ acts by $\frac{k\ell}{2}$ and let 
\[ \ddot{\mathbb{C}}:=\begin{cases}
(\mbbC \setminus \mathbb{Z}) \cup r\mathbb{Z} & \text{ if $\ell$ is even}\\
(\mbbC \setminus \frac{1}{2}\mathbb{Z}) \cup \frac{r}{2}\mathbb{Z} & \text{ if $\ell$ is odd}.
\end{cases}
 \]
Every irreducible module in $\mcC$ is isomorphic to exactly one of the following \cite[Sections 5 \& 10]{CGP}:
\begin{itemize}
\item $L_i \otimes \mathbb{C}_{\frac{k\ell}{2}}^H$, for $i=0,...,r-2,k \in \mathbb{Z}$.
\item $M_{\alpha}$ for $\alpha+1 \in \ddot{\mathbb{C}}$.
\end{itemize}
The $M_{\alpha}$ for $\alpha+1 \in \ddot{\mathbb{C}}$ are the irreducible Verma modules of typical highest weight. Their projective cover in $\mcC_m$ for $m \in \mathbb{N}$ (recall now that $n=\mathrm{rank}(\mathfrak{sl}_2)=1$) is the generalized Verma module $\GVerma{\alpha}{m}$ of degree $m$. The projective covers of the modules $L_i \otimes \mathbb{C}_{\frac{k\ell}{2}}^H$, for $i=0,...,r-2,k \in \mathbb{Z}$ in $\mcC_0$ were explicitly determined for even roots of unity in \cite[Section 6]{CGP} where they were shown to take the form $P_i \otimes \mathbb{C}_{\frac{k \ell}{2}}^H$ where $P_i$ is generated by a ``dominant" weight vector of weight $i$. That is, a vector $v$ satisfying $(FE)^2v=0$ and $Hv=iv$. In what follows, we consider the structure of modules generated by dominant vectors with non-semisimple $H$-action. We refer to a vector $v$ in a module $V \in \mcC$ as dominant of weight $i$ and degree $m\in \mathbb{N}$ if
\[ (FE)^2v=0, \qquad (H-i)^{m+1}v=0, \qquad (H-i)^{m}v\not =0. \]
We consider now the structure of submodules generated by such vectors.
\begin{proposition}\label{Prop:projgen}
Let $v \in V$ be dominant of weight $i \in  \{0,1,...,r-2\}$ and degree $m$ and let $j=r-2-i$. Then the submodule of $V$ generated by $v$ has a basis given by the $2(m+1)r$ vectors 
\[\w_{i-2k,s}^T,\quad \w_{i-2k,s}^S, \quad \w_{j-r-2k',s}^L, \quad \w_{r-j+2k',s}^R\] with $k \in \{0,...,i\}$, $k' \in \{0,...,j\}$, and $s\in \{0,...,m\}$ defined by
\begin{align}
\label{eq:proj1}\w_{i,m}^T:=v, \quad \w_{r-j,m}^R:=E\w_{i,m}^T, \quad \w_{i,m}^S:=F\w_{r-j,m}^R, \quad \w_{j-r,m}^L:=F^{i+1}\w_{i,m}^T,
\end{align}
\vspace{-2em}
\begin{align}
\label{eq:proj2}\w_{i-2k,m}^T&:=F^k\w_{i,m}^T & &\mathrm{and} & \w_{i-2k,m}^S&:=F^k\w_{i,m}^S & &\text{for $k \in \{0,...,i\}$,}\\
\label{eq:proj3}\w_{j-r-2k,m}^L&:=F^k\w_{j-r,m}^L, & &\mathrm{and} &  \w_{r-j+2k,m}^R&:=E^k\w_{r-j,m}^R, & &\text{for $k \in \{0,...,j\}$,}
\end{align}
\vspace{-2em}
\begin{align}
\label{eq:proj4}\w_{t,s}^X:=(H-t)^{m-s}\w_{t,m}^X
\end{align}
where $X \in \{T,S,L,R\}$ and $t$ takes values as in Equations \eqref{eq:proj2}-\eqref{eq:proj3}.
These vectors satisfy the following relations whenever the vectors involved are defined:
\begin{align}
\label{eq:proj5}(H-t)^{s+1}\w_{t,s}^X&=0, & H\w_{t,s}^X&=\w_{t,s-1}^X+t\w_{t,s}^X, & &\mathrm{for} \; X\in \{T,S,L,R\}  \\
\label{eq:proj6}\; & &K\w_{t,s}^X&=q^t\sum\limits_{k=0}^s \left(\frac{2 \pi \mathsf{i}}{\ell}\right)^k\w^X_{t,s-k} & &\mathrm{for} \; X \in \{T,S,L,R\} \\
\label{eq:proj7}E\w_{t,s}^R&=\w_{t+2,s}^R, & F\w_{t,s}^X&=\w_{t-2,s}^X & &\mathrm{for} \; X \in \{T,S,L\} 
\end{align}
\vspace{-2em}
\begin{align}
\label{eq:proj8} (FE)^2\w_{i,s}^T=0 , \quad F\w_{-i,s}^T=\w_{j-r,s}^L \quad E\w_{i,s}^T=\w_{r-j,s}^R, \quad F\w_{r-j,s}^R=\w_{i,s}^S
\end{align}
\vspace{-2em}
\begin{align}\label{eq:proj9}
E\w_{j-r,s}^L&=\w_{-i,s}^S+\sum\limits_{t=0}^{\lfloor \frac{s-1}{2} \rfloor} \left(\frac{2 \pi \mathsf{i}}{\ell}\right)^{2t+1}\frac{2[i+1]}{q-q^{-1}} \w_{-i,s-2t-1}^T,
\end{align}
\vspace{-1em}
\begin{align}\label{eq:proj10}
E\w_{j+r,s}^R=E\w_{i,s}^S=F\w_{-i,s}^S=F\w_{-j-r,s}^L=0,
\end{align}
\vspace{-2em}
\begin{align}
\label{eq:proj11}E\w_{i-2k,s}^T=\w_{i-2k+2,s}^S&+\sum\limits_{t=0}^{\lfloor \frac{s-1}{2} \rfloor}\left(\frac{2\pi \mathsf{i}}{\ell}\right)^{2t+1} \frac{q^{i-k+1}+q^{-i+k-1}}{q-q^{-1}}[k]\w_{i-2k+2,s-2t-1}^T\\
\nonumber&+\sum\limits_{t=0}^{\lfloor \frac{s}{2} \rfloor} \left(\frac{ 2 \pi \mathsf{i}}{\ell}\right)^{2t} [k][i-k+1]\w_{i-2k+2,s-2t}^T,
\end{align}
\vspace{-1em}
\begin{align}\label{eq:proj12}
 E\w_{i-2k,s}^S&=\sum\limits_{t=0}^{\lfloor \frac{s-1}{2} \rfloor}\left(\frac{2\pi \mathsf{i}}{\ell}\right)^{2t+1} \frac{q^{i-k+1}+q^{-i+k-1}}{q-q^{-1}}[k]\w_{i-2k+2,s-2t-1}^S\\
\nonumber&+\sum\limits_{t=0}^{\lfloor \frac{s}{2} \rfloor} \left(\frac{ 2 \pi \mathsf{i}}{\ell}\right)^{2t} [k][i-k+1]\w_{i-2k+2,s-2t}^S,
\end{align}
\vspace{-1em}
\begin{align}
\label{eq:proj13}E\w_{j-r-2k,s}^L&=\sum\limits_{t=0}^{\lfloor \frac{s-1}{2}\rfloor} \left(\frac{2\pi \mathsf{i}}{\ell}\right)^{2t+1}\frac{q^k+q^{-k}}{q-q^{-1}}[j-k+1]\w_{j-r-2k+2,s-2t-1}^L\\
\nonumber&-\sum\limits_{t=0}^{\lfloor \frac{s}{2}\rfloor}\left(\frac{2\pi \mathsf{i}}{\ell}\right)^{2t} [j-k+1][k]\w_{j-r-2k+2,s-2t}^L,
\end{align}
\vspace{-1em}
\begin{align}\label{eq:proj14}
F\w_{r-j+2k,s}^R&=\sum\limits_{t=0}^{\lfloor \frac{s-1}{2} \rfloor} \left(\frac{2 \pi \mathsf{i}}{\ell} \right)^{2t+1} \frac{q^{j-k+1}+q^{-j+k-1}}{q-q^{-1}} [k]\w_{r-j+2k+2,s-2t-1}^R\\
\nonumber&-\sum\limits_{t=0}^{\lfloor \frac{2}{2}\rfloor} \left(\frac{2\pi \mathsf{i}}{\ell} \right)^{2t} [j-k+1][k]\w_{r-j-2k+2,s-2t}^R.
\end{align}
\end{proposition}
\begin{proof}
We first prove that Equations \eqref{eq:proj5}-\eqref{eq:proj8} hold. It follows immediately from Equation \eqref{eq:proj4} that 
\[ \w_{t,s-1}^X=(H-t)^{m-s+1}\w_{t,m}^X=(H-t)\w_{t,s}^X\]
so $H\w_{t,s}^X=\w_{t,s-1}^X+t\w_{t,s}^X$, as desired. By definition, we have
\[ (H-t)^{m+1}\w_{t,m}^T =0, \qquad \mathrm{and} \qquad (H-t)^m\w_{t,m}^T \not = 0,\]
so $(H-t)^{s+1}\w_{t,s}^T=(H-t)^{s+1}(H-t)^{m-s}\w_{t,m}^T=0$ and Equation \eqref{eq:proj5} holds. Recall that $K=q^H$ as operations on $V$ so
\[ K=q^H=q^{t+H-t}=q^tq^{H-t}=q^t\sum\limits_{k=0}^{\infty} \left(\frac{2 \pi \mathsf{i}}{\ell}\right)^k(H-t)^k \]
as an operator on $V$. In particular,
\begin{align*} K\w_{t,s}^X=q^t \sum\limits_{k=0}^s \left(\frac{2 \pi \mathsf{i}}{\ell}\right)^k (H-t)^k\w_{t,s}=q^t \sum\limits_{k=0}^s \left(\frac{2 \pi \mathsf{i}}{\ell}\right)^k \w_{t,s-k}
\end{align*}
where we have used Equation \eqref{eq:proj4} in the last equality, proving that Equation \eqref{eq:proj6} holds. Observe now that the following relations hold by Equation \eqref{eq:HX}:
\begin{align}
\label{eq:E(H-t)}E\w_{t,s}^X=E(H-t)^{m-s}\w_{t,m}^X=(H-(t+2))^{m-s}E\w_{t,m}^X,\\
\label{eq:F(H-t)}F\w_{t,s}^X=F(H-t)^{m-s}\w_{t,m}^X=(H-(t-2))^{m-s}F\w_{t,m}^X.
\end{align}
Hence, for any $k \in \{0,...,j\}$ we have 
\begin{align*} E\w_{r-j+2k,s}^R&=(H-(r-j+2k+2))^{m-s}E\w_{r-j+2k,m}^R\\
&=(H-(r-j+2(k+1)))^{m-s}\w_{r-j+2(k+1),m}^R=\w_{r-j+2k+2,s}^R
\end{align*}
where we have used Equations \eqref{eq:proj1} and \eqref{eq:proj3} so $E\w_{t,s}^R=\w_{t+2,s}^R$. The proof that $F\w_{t,s}^X=\w_{t-2,s}^X$ for $X \in \{T,S,L\}$ is similar, so Equation \eqref{eq:proj7} holds. We now show that Equation \eqref{eq:proj8} holds. Notice first that by Equations \eqref{eq:proj1}, \eqref{eq:F(H-t)}, and \eqref{eq:E(H-t)}, we have
\begin{align*} F\w_{-i,s}^T&=(H+i+2)^{m-s}F\w_{-i,m}^T & E\w_{i,s}^T&=(H-i-2)^{m-s}E\w_{i,m}^T\\
&=(H+i+2)^{m-s}F^{i+1}\w^T_{i,m} & &=(H-i-2)^{m-s}\w_{r-j,m}^R\\
&=(H-(j-r))^{m-s}\w_{j-r,m}^L & &=(H-(r-j))^{m-s}\w_{r-j,m}^R \\
&=\w_{j-r,s}^L & &=\w_{r-j,s}^R
\end{align*}
The relation $F\w_{r-j,s}^R=\w_{i,s}^S$ also follows by Equation \eqref{eq:proj1} and \eqref{eq:F(H-t)} by an identical argument. Finally, by Equations \eqref{eq:E(H-t)} and \eqref{eq:F(H-t)} we have
\[ (EF)^2\w_{i,s}^T=(H-t)^{m-s}(EF)^2\w_{t,m}^T=0\]
 so Equation \eqref{eq:proj8} holds. Consider now the relation
\begin{equation}\label{eq:EF}
EF^k=F^kE+\sum\limits_{t=0}^{k-1}F^{k-1-t}\frac{K-K^{-1}}{q-q^{-1}}F^t
\end{equation}
which can be easily shown by induction, and observe that
\begin{align*}
(K-K^{-1})\w_{t,s}^X&=q^t\sum\limits_{k=0}^s \left(\frac{2\pi \mathsf{i}}{\ell}\right)^k \w_{t,s-k}^X-q^{-t}\sum\limits_{k=0}^s\left(\frac{-2\pi \mathsf{i}}{\ell}\right)^k\w_{t,s-k}\\
&=\sum\limits_{k=0}^s \left(\frac{2\pi \mathsf{i}}{\ell}\right)^k(q^t-(-1)^kq^{-t})\w_{t,s-k}^X.
\end{align*}
Applying these two equations gives us the following for $X=T,S,L$:
\begin{align}
\nonumber E\w_{i-2k,s}^X=EF^k\w_{i,s}^X=F^kE\w_{i,s}^X&+\sum\limits_{u=0}^{k-1}F^{k-1-u}\frac{K-K^{-1}}{q-q^{-1}}F^u\w_{i,s}^X\\
\nonumber =F^kE\w_{i,s}^X&+\sum\limits_{u=0}^{k-1}F^{k-1-u}\sum\limits_{t=0}^s \left(\frac{2\pi \mathsf{i}}{\ell}\right)^m\frac{(q^{i-2u}-(-1)^mq^{-i+2u})}{q-q^{-1}}F^u\w_{i,s-t}^X\\
\label{eq:projcomp}=F^kE\w_{i,s}^X&+\sum\limits_{t=0}^s \left(\frac{2\pi \mathsf{i}}{\ell}\right)^m\sum\limits_{u=0}^{k-1}\frac{(q^{i-2u}-(-1)^mq^{-i+2u})}{q-q^{-1}}F^{k-1}\w_{i,s-t}^X\\
\nonumber =F^kE\w_{i,s}^X&+\sum\limits_{t=0}^{\lfloor \frac{s-1}{2} \rfloor}\left(\frac{2\pi \mathsf{i}}{\ell}\right)^{2t+1} \frac{q^{i-k+1}+q^{-i+k-1}}{q-q^{-1}}[k]F^{k-1}\w_{i,s-2t-1}^X\\
\nonumber &+\sum\limits_{t=0}^{\lfloor \frac{s}{2} \rfloor} \left(\frac{ 2 \pi \mathsf{i}}{\ell}\right)^{2t} [k][i-k+1]F^{k-1}\w_{i,s-2t}^X.
\end{align}
Here we have broken our sum over $t$ into odd and even values and applied the easily verified identities
\[ \sum\limits_{u=0}^{k-1}\frac{q^{i-2u}+q^{-i+2u}}{q-q^{-1}}=\frac{q^{i-k+1}+q^{-i+k-1}}{q-q^{-1}}[k] \quad \mathrm{and} \quad \sum\limits_{u=0}^{k-1}\frac{q^{i-2u}-q^{-i+2u}}{q-q^{-1}}=[k][i-k+1]. \]
Taking $X=T$ and recalling that $F^kE\w_{i,s}^T=F^k\w_{r-j,s}^R=\w_{i-2k+2,s}^S$ gives us Equation \eqref{eq:proj11}. Notice that we can also choose $k=i+1$ in the above computation to obtain Equation \eqref{eq:proj9}:
\begin{align*}
E\w_{j-r,s}^L=EF^{i+1}\w_{i,s}^T&=\w_{-i,s}^S+\sum\limits_{t=0}^{\lfloor \frac{s-1}{2} \rfloor} \left(\frac{2 \pi \mathsf{i}}{\ell}\right)^{2t+1}\frac{2[i+1]}{q-q^{-1}} \w_{-i,s-2t-1}^T.
\end{align*}
Further, $X=S$ in Equation \eqref{eq:projcomp} provides
\begin{align*}
 E\w_{i-2k,s}^S=F^kE\w_{i,s}^S &+\sum\limits_{t=0}^{\lfloor \frac{s-1}{2} \rfloor}\left(\frac{2\pi \mathsf{i}}{\ell}\right)^{2t+1} \frac{q^{i-k+1}+q^{-i+k-1}}{q-q^{-1}}[k]\w_{i-2k+2,s-2t-1}^S\\
&+\sum\limits_{t=0}^{\lfloor \frac{s}{2} \rfloor} \left(\frac{ 2 \pi \mathsf{i}}{\ell}\right)^{2t} [k][i-k+1]\w_{i-2k+2,s-2t}^S.
\end{align*}
where $F^kE\w_{i,S}^S=F^kEFE\w_{i,s}^T=0$ as $(FE)^2\w_{i,s}^T=0$ so Equation \eqref{eq:proj12} holds. To prove Equation \eqref{eq:proj13}, we first show that $F\w_{-i,s}^S=0$ starting with $F^{r-2}\w_{i,s}^S=0$ for all $s=0,...,m$. Indeed, we have
\begin{align*} F\w_{-j+r+2,s}^R=FE\w_{-j+r,s}^R&=EF\w_{-j+r,s}^R-\frac{K-K^{-1}}{q-q^{-1}}\w_{-j+r,s}^R\\
&=E\w_{i,s}^S-\sum\limits_{k=0}^s \left(\frac{2 \pi \mathsf{i}}{\ell}\right)^k \frac{(q^{-j+r}-(-1)^kq^{r-j})}{q-q^{-1}} \w_{-j+r,s-k}^R.
\end{align*}
Recall that $(FE)\w_{i,s}^S=(FE)^2\w_{i,s}^T=0$ so acting on both sides by $F^{r-1}$ gives
\[ 0=-F^r\w_{-j+r+2,s}^R=\sum\limits_{k=0}^s \left(\frac{2 \pi \mathsf{i}}{\ell}\right)^k \frac{(q^{-j+r}-(-1)^kq^{r-j})}{q-q^{-1}} F^{r-2}\w_{i,s-k}^S. \]
If the terms $F^{r-2}\w_{i,s-k}^S$, $k=0,...,s$ are non-zero, they lie in different generalized weight spaces so they are linearly independent, contradicting the above equation. Hence, $F^{r-2}\w_{i,s}^S=0$ for all $s=0,...,m$ which will form the base step of an induction argument. Suppose now that $F^k\w_{i,s}^S=0$ for $r-2 \leq k \leq i+2$ and in particular, $EF^{k}\w_{i,s}^S=0$ for all $s=0,...,m$. We can now use Equation \eqref{eq:projcomp} to show that a particular linear combination of $\{F^{k-1}\w_{i,s-t}^S\}_{t=0}^s$ with non-zero coefficients vanishes. If these terms are non-zero they are linearly independent as they lie in different generalized weight spaces which gives a contradiction, so $F^{k-1}\w_{i,s}^S=0$ for all $s=0,...,m$ . Hence, the induction step holds and $F\w_{-i,s}^S=F^{i+1}\w_{i,s}^S=0$. Equation \eqref{eq:proj13} now follows from a routine application of Equation \eqref{eq:projcomp}. It is not difficult to see from Equations \eqref{eq:proj9}, \eqref{eq:proj12}, and \eqref{eq:proj13} that $E^{r-1}F^{r-1}\w_{-j-r,s}^L$ is given by
\[ E^{r-1}F^{r-1}\w_{-j-r,s}^L=(-1)^j[j]!^2[i]!^2\w_{i,s}^S + \sum\limits_{t=1}^s \lambda_t \w_{i,s-t}^S\]
for some scalars $\lambda_1,...,\lambda_s$ so we will have $E\w_{i,s}^S=0$ if $E\w_{i,k}^S=0$ for all $0 \leq k <s$. We therefore only need to prove the base case $E\w_{i,0}^S=0$ and we will have $E\w_{i,s}^S=0$ for $s=0,...,m$ by induction. Notice however that $\w_{i,0}^T$ is a dominant weight vector of weight $i$, so the submodule it generates satisfies the relations given in \cite[Proposition 6.2]{CGP}. In particular, we have $E\w_{i,0}^S=0$, as desired. We will neglect the proof of Equation \eqref{eq:proj14} as it is similar to the proof for Equation \eqref{eq:proj13}. One simply obtains an equation analogous to Equation \eqref{eq:projcomp} derived from the relation
\[ FE^k=E^kF-\sum\limits_{t=0}^{k-1}E^{k-1-t}\frac{K-K^{-1}}{q-q^{-1}}E^t \]
and uses $E\w_{i,s}^S=0$ which was just shown. The only equation which remains to prove is Equation \eqref{eq:proj10} of which we have already shown $E\w_{i,s}^S=F\w_{-i,s}^S=0$. The argument for $E\w_{j+r,s}^R=0$ is similar to that for $F\w_{-i,s}^S=0$ and $F\w_{-j-r,s}^L=F^r\w_{i,s}^T=0$.\\

It only remains to show that
 \[\w_{i-2k,s}^T,\quad \w_{i-2k,s}^S, \quad \w_{j-r-2k',s}^L, \quad \w_{r-j+2k',s}^R\] with $k \in \{0,...,i\}$, $k' \in \{0,...,j\}$, and $s\in \{0,1,...,m\}$ form a basis of the submodule generated by $V$. It is clear from Equations \eqref{eq:proj1} - \eqref{eq:proj14} that the submodule is spanned by the vectors so we need only establish their linear independence. Suppose
\[\sum\limits_{\stackrel{k \in \{0,...,j\}}{ s \in \{0,...,m\}}}c^L_{k,s}\w_{j-r-2k,s}^L+\sum\limits_{\stackrel{k \in \{0,...,j\}}{ s \in \{0,...,m\}}}c^R_{k,s} \w_{r-j+2k,s}^R+\sum\limits_{\stackrel{k \in \{0,...,i\}}{ s \in \{0,...,m\}}}c^T_{k,s} \w_{i-2k,s}^T+\sum\limits_{\stackrel{k \in \{0,...,i\}}{ s \in \{0,...,m\}}}c^S_{k,s}\w_{i-2k,s}^S=0. \]
It can be seen by a standard argument that $c_{k,s}^L=c_{k,s}^S=0$ for all $k,s$ since they have distinct weights, so we can isolate the coeffecients through a suitable action of $H$. Consider now the remaining terms and act on both sides by $F^{i+1}$:
\[ 0=\sum\limits_{\stackrel{k \in \{0,...,i\}}{ s \in \{0,...,m\}}}c^T_{k,s} F^{i+1}\w_{i-2k,s}^T+\sum\limits_{\stackrel{k \in \{0,...,i\}}{ s \in \{0,...,m\}}}c^S_{k,s}F^{i+1}\w_{i-2k,s}^S =\sum\limits_{\stackrel{k \in \{0,...,i\}}{ s \in \{0,...,m\}}}c^T_{k,s} \w_{j-r-2k,s}^L\]
where we have used the fact that $F^{i+1}\w_{i-2k,s}^S=0$ and $F^{i+1}\w_{i-2k,s}^T=\w_{j-r-2k,s}^L$ for all $k,s$. We can again use a suitable action of $H$ to obtain $c^T_{k,s}=0$ for all $k,s$ and we obtain
\[ \sum\limits_{\stackrel{k \in \{0,...,i\}}{ s \in \{0,...,m\}}}c^S_{k,s}\w_{i-2k,s}^S=0\]
so $c^S_{k,s}=0$ follows by the same argument.
\end{proof}

We will denote by $\Proj{i}{m}$ the module generated by a single dominant vector $v=\w_{i,m}^T$ of generalized weight $i \in \{0,1,....,r-2\}$ and degree $m$. By Proposition \ref{Prop:projgen}, this module is indecomposable with dimension $2(m+1)r$ and has a vector space basis given by
 \[\w_{i-2k,s}^T,\quad \w_{i-2k,s}^S, \quad \w_{j-r-2k',s}^L, \quad \w_{r-j+2k',s}^R\] with $k \in \{0,...,i\}$, $k' \in \{0,...,j\}$, and $s\in \{0,1,...,m\}$, subject to Equations \eqref{eq:proj1} - \eqref{eq:proj14}. Verifying that Equations \eqref{eq:sl21} \& \eqref{eq:sl22} are satisfied is an easy but tedious computation. It was shown in \cite[Proposition 6.2]{CGP} that $\Proj{i}{0}$ is projective. Their proof relied on the semisimple action of $K$ and $H$, so we will instead use Proposition \ref{Prop:tiltingproj} to generalize this statement. 

\begin{theorem}\label{Thm:projcover}
The module $\Proj{i}{m} \otimes \mathbb{C}^H_{\frac{k\ell}{2}}$ is the projective cover of $L_i \otimes \mathbb{C}^H_{\frac{k\ell}{2}}$ in $\mcC_m$.
\end{theorem}
\begin{proof}
We begin by showing that $\Proj{i}{m}$ is projective. First, observe from Equation \eqref{eq:proj10} that the vector $\w_{j+r,m}^R$ is highest weight with weight $j+r$ and degree $m$. Therefore, there exists a surjection $\varphi:\GVerma{j+r}{m} \twoheadrightarrow \langle \w_{j+r,m}^R\rangle$ onto the submodule $\langle \w_{j+r,m}^R \rangle$ of $\Proj{i}{m}$ generated by $\w_{j+r,m}^R$. It follows from Proposition \ref{Prop:projgen} that this submodule has a basis given by 
\[ \w_{i-2k,s}^S,  \quad \w_{r-j+2k',s}^R\] 
with $k \in \{0,...,i\}$, $k' \in \{0,...,j\}$, and $s\in \{0,1,...,m\}$. Hence, $\dim \GVerma{j+r}{m}=\dim \langle \w_{j+r,m}^R \rangle=(m+1)r$ so $\varphi$ is in fact an isomorphism. Now $\Proj{i}{m}/\langle \w_{j+r,m}^R \rangle$ is a highest weight module generated by the image of $\w_{i,m}^T$ and it is clear again from Proposition \ref{Prop:projgen} that this module has a basis given by the images of
\[ \w_{i-2k,s}^T,  \quad \w_{j-r-2k',s}^L\]
so $\dim \Proj{i}{m}/\langle \w_{j+r,m}^R \rangle= \dim \GVerma{i}{m}=(m+1)r$. That is, $\Proj{i}{m}/\langle \w_{j+r,m}^R \rangle \cong \GVerma{i}{m}$ so $\Proj{i}{m}$ admits a degree m standard filtration of length 2. 

Consider now the submodule $\langle \w_{-j-r,m}^L \rangle$ of $\Proj{i}{m}$. This is a lowest weight module generated by the lowest weight vector $\w_{-k-r,m}^L$ of degree m and contains the highest weight vector $\w_{i,m}^S$ of weight $i$ and degree $m$. We claim now that the dual $\check{\langle \w_{-j-r,m}^L \rangle}$ is a highest weight module. Indeed, this module has vector space basis given by the dual basis $\{\delta_{i-2k,s}, \delta_{j-r-2k',s}\},$ where $k \in \{0,...,i\}, k'\in \{0,...,j\}$ and $s \in \{0,...,m\}$. Here we are using the notation $\delta_{i-2k,s}:=(\w_{i-2k,s}^S)^*$ for $k \in \{0,...,i\}$ and $\delta_{j-r-2k,s}:=(\w_{j-r-2k,s}^L)^*$ for $k \in \{0,...,j\}$. The $\overline{U}_q^H(\mfg)$-action satisfies (recall Definition \ref{Def:dual}):
\begin{align*}
(H \cdot \delta_{t,s})(x)&=\delta_{t,s}( H \cdot x), & (K \cdot \delta_{t,s})(x)&=\delta_{t,s}(K \cdot x),\\
(E \cdot \delta_{t,s})(x)&= \delta_{t,s}(-KF \cdot x), & (F \cdot \delta_{t,s})(x)&=\delta_{t,s}(-EK^{-1} \cdot x).\\
\end{align*}
It follows easily from these relations that
\[ E \cdot \delta_{i,s}=0, \quad (H-t) \cdot \delta_{t,s}=\delta_{t,s+1}, \quad F \cdot \delta_{t,s}=\delta_{t-2,s} \]
In particular, $\delta_{i,0}$ is a generalized highest weight vector with weight $i$ and degree $m$. Further, it follows from the structure of $\langle \w_{-j-r,m}^L \rangle$ that for all basis vectors $\w_{t,s} \in \langle \w_{-j-r,m}^L \rangle$, there exists some $X \in \overline{U}_q^H(\mfg)$ such that $X \w_{t,s}=\w_{i,0}^S$. Let $X' \in \overline{U}_q^H(\mfg)$ be such that $X=S(\omega(X'))$. Then,
\[ (X' \cdot \delta_{i,0})( x)=\delta_{i,0}( X \cdot x)\]
so $(X' \cdot \delta_{i,0})=\delta_{t,s}$. That is, $\delta_{i,0}$ generates $\check{\langle  \w_{j-r-2k,m}^L \rangle}$ so it is a generalized highest weight module with highest weight $i$ and degree $m$. We therefore have
\[  \check{\langle  \w_{-j-r,m}^L \rangle} \cong \GVerma{i}{m}\]
as they have the same dimension, so $ \langle  \w_{j-r-2k,m}^L \rangle \cong \check{M}_i^m$. One can similarly show that $\Proj{i}{m}/\langle \w_{-j-r,m}^L \rangle \cong \check{M}_{j+r}^m$, that is, $\Proj{i}{m}$ admits a degree $m$ costandard filtration so it is in fact a tilting module in $\mcC_m$ and therefore projective by Proposition \ref{Prop:tiltingproj}. It is easy to see that $L_i$ is the unique simple quotient of $P_i$ so $L_i \otimes \mathbb{C}_{\frac{k\ell}{2}}$ is the unique simple quotient of $P_i \otimes \mathbb{C}_{\frac{k \ell}{2}}$ as $\mbbC_{\frac{k \ell}{2}}$ is invertible and tensoring by an invertible object gives an exact functor. Therefore, $\Proj{i}{m} \otimes \mathbb{C}_{\frac{k\ell}{2}}$ is the projective cover of $L_i \otimes \mathbb{C}_{\frac{k \ell}{2}}$ in $\mcC_m$ by \cite[Lemma 3.6]{Kr}.

%Consider again the Verma module $\GVerma{i}{m}$. It can be easily deduced from $\Proj{i}{m}/\langle \w_{j+r,m}^R \rangle \cong \GVerma{i}{m}$ and Proposition \ref{Prop:projgen} that $\GVerma{i}{m}$ has a vector space basis $v_{i-2k,s}$ where $k \in \{0,...,r-1\}$, $s \in \{0,...,m\}$, and relations
%\begin{align*}
% (H-(i-2k))^tv_{i-2k,s}&=v_{i-2k,s-t}, \qquad  Fv_{i-2k,s}=v_{i-2(k+1),s},  \\
%Kv_{i-2k,s}&=q^{i-2k}\sum\limits_{t=0}^s \left(\frac{2\pi \mathsf{i}}{\ell}\right)^t v_{i-2k,s-t}, \\
%Ev_{i-2k,s}&=\sum\limits_{t=0}^{\lfloor \frac{s-1}{2} \rfloor}\left(\frac{2\pi \mathsf{i}}{\ell}\right)^{2t+1} \frac{q^{i-k+1}+q^{-i+k-1}}{q-q^{-1}}[k]v_{i-2k+2,s-2t-1}\\
%\nonumber&+\sum\limits_{t=0}^{\lfloor \frac{s}{2} \rfloor} \left(\frac{ 2 \pi \mathsf{i}}{\ell}\right)^{2t} [k][i-k+1]v_{i-2k+2,s-2t}.
%\end{align*}

\end{proof}

\end{document}